\documentclass[12pt, twoside, leqno]{article} 
\usepackage[latin1]{inputenc} 
\usepackage[english]{babel} 
\usepackage[T1]{fontenc}
\usepackage{bm}       
\usepackage{amsmath,amsthm,amssymb} 
\usepackage{enumerate}    
\usepackage{amsfonts}   
\usepackage{amsthm,booktabs}  
\newcommand{\overbar}[1]{\mkern 1.5mu\overline{\mkern-1.5mu#1\mkern-1.5mu}\mkern 1.5mu}

\newtheorem{thm}{Theorem}[section] 
\newtheorem{prop}[thm]{Proposition}

 \theoremstyle{definition} 
 
 \theoremstyle{definition}

\begin{document}
\baselineskip=17pt

\title{Properties of cuspidal divisor class numbers of non-split Cartan modular curves}
 
\author{Pierfrancesco Carlucci\\
Dipartimento di Matematica\\
 Universit\'a degli studi di Roma Tor Vergata\\
  Via della Ricerca Scientifica 1, 00133, Rome, Italy\\ 
E-mail: pieffecar@libero.it}
\date{30 May, 2016}
 
\maketitle


\renewcommand{\thefootnote}{}



\renewcommand{\thefootnote}{\arabic{footnote}}
\setcounter{footnote}{0}
\begin{abstract}
Let $ \mathfrak{C}^+_{ns}(p) $ be the Cuspidal Divisor Class Group of the modular curves $X^+_{ns}(p) $ associated to the normalizer of a non-split Cartan subgroup of level $ p$. I study the $ p-$primary part of $ \mathfrak{C}^+_{ns}(p) $ and estimate the order of growth of $  |\mathfrak{C}^+_{ns}(p)| $. 
\end{abstract}

\section{Introduction} 

Let $ p $ be a prime and let $ X^+_{ns}(p^k) $ be the modular curve associated to the normalizer of a non-split Cartan subgroup of level $p^k$. In \cite{Io} we describe the Cuspidal Divisor Class Group $  \mathfrak{C}^+_{ns}(p^k) $  on $X^+_{ns}(p^k)$ as a module over the group ring $ \mathbb{Z}[(\mathbb{Z}/p^k\mathbb{Z})^*/\{\pm 1\}] $. Let $ w $ be a generator of $ H=(\mathbb{Z}/p\mathbb{Z})^*/\{\pm 1\} $ and let $ \omega $ be a generator of the character group $ \hat{\mathbb{F}_{p^2}^*} $. Define $ d=\displaystyle\frac{12}{\gcd(12,p+1)} $, the group ring $ R= \mathbb{Z}[H] $, the ideals: $$ R_0 := \Big{\{} \sum b_jw^j \in R \mbox{ such that } \deg \left( \sum b_jw^j\right)=\sum b_j=0 \Big{\}}, $$
$$  R_d := \Big{\{} \sum b_jw^j \in R \mbox{ such that } d \mbox{ divides } \deg\left(\sum b_jw^j\right)=\sum b_j \Big{\}}, $$
the Stickelberger element:
$$ \theta =\displaystyle\frac{p}{2} \sum_{i=0}^{\frac{p-3}{2}}\sum_{ { \begin{scriptsize} \begin{array}{c} x \in \mathbb{F}^*_{p^2}/\{\pm 1\}  \\  \pm x^{p+1} =w^i \end{array} \end{scriptsize} }} B_2 \left( \left\langle \frac{\frac{1}{2}(\mbox{Tr}(x)) }{p} \right\rangle \right) w^{-i} \in \mathbb{Q}[H]  $$
and the generalized Bernoulli number:
$$ B_{2,\chi} =  \sum_{x \in \mathbb{F}_{p^2}^* / \{\pm 1\}} B_2 \left( \left\langle \frac{\frac{1}{2}\mbox{Tr}(x)}{p} \right\rangle \right) \chi(x). $$
Specializing \cite[Theorem 7.1]{Io} and \cite[Theorem 7.4]{Io} to the case $k=1$ we obtain:

\begin{thm}\label{robamia} The Cuspidal Divisor Class Group on $ X^+_{ns}(p) $ is a module over $R$ and we have the following isomorphism: 
$$ \mathfrak{C}^+_{ns}(p) \cong R_0 / R_d \theta. $$
Moreover we have:
$$ |\mathfrak{C}^+_{ns}(p)| = \displaystyle \frac{24}{(p-1)\gcd(12,p+1)}\prod_{j=1}^{\frac{p-3}{2}}\frac{p}{2}B_{2,\omega^{(2p+2)j}}.
 $$
\end{thm}

From the previous theorem we deduce two results both having a counterpart in cyclotomic field theory. \\

\noindent \textbf{Theorem \ref{Ordinedigrandezza}} \textit{We have}:  $$ \ln |\mathfrak{C}^+_{ns}(p)| = p\ln p + \Theta(p).$$

The paper ends up with a modular analogue of Mazur-Wiles  \cite[pag. 300]{Washington}, Herbrand \cite[pag. 101]{Washington} and Ribet \cite[pag. 342]{Washington} theorems for cyclotomic fields. We have a similar piece by piece description of the $ p$-Sylow part $  \mathcal{C}_p $ of $ \mathfrak{C}^+_{ns}(p) $. Let $A$ be the $p$-Sylow subgroup of the ideal class group of $ \mathbb{Q}(\zeta_p) $. A corollary of Mazur-Wiles theorem states that:
$$ |A(\omega^i)| = p\mbox{-part of } B_{1,\omega^{-i}} \mbox{ (}i\not=1\mbox{ mod }p-1, i \mbox{ odd}).  $$ 
A reformulation of Kubert and Lang Theorems 4.2 and 4.3 of \cite[Chapter 5]{KL}, enables us to deduce that, as in the cycltomic field theory, $ |\mathcal{C}_p(\omega^{2j})| $ are strictly related to the $p$-parts of certain generalized Bernoulli numbers $ B_{2,\omega^{4j}} $.
 Usually we expect that $ \mathcal{C}_p \cong (\mathbb{Z}/p\mathbb{Z})^{[\frac{p}{4}]-1} $ and exceptions occur only when $p$ is an irregular prime. More precisely: 
 \\    
 \\ \textbf{Theorem \ref{analogomodulare}} 
 \textit{ $ \mbox{ord}_p(|\mathfrak{C}^+_{ns}(p)|)= [\frac{p}{4}]-1 $ if and only if $ p $ is a regular prime or $ p \equiv 1 \mbox{ mod }4 $ is irregular and $ p $ does not divide the numerator of any Bernoulli number $ b_{4j+2} $ for $ j \le \frac{p-5}{4}$. \\   
If  $ \mbox{ord}_p(|\mathfrak{C}^+_{ns}(p)|)= [\frac{p}{4}]-1 $ we have $ \mathcal{C}_p \cong (\mathbb{Z}/p\mathbb{Z})^{[\frac{p}{4}]-1} $.\\
 If $ p \equiv 1 \mbox{ mod } 4 $ and  $ \mbox{ord}_p(|\mathfrak{C}^+_{ns}(p)|)> [\frac{p}{4}]-1 $ then $ \mathfrak{C}^+_{ns}(p) $ contains an element of order $ p^2 $ .\\
 If $ p \equiv 3 \mbox{ mod } 4 $ and  $ \mbox{ord}_p(|\mathfrak{C}^+_{ns}(p)|)> [\frac{p}{4}]-1 $ then $ \mathfrak{C}^+_{ns}(p)$ contains  an element of order $ p^2 $ if and only if $ p $ divides $ b_{4j+2} $ for some $ j \le \frac{p-7}{4} $.\\  
 If $ p \equiv 3 \mbox{ mod } 4 $ then $  \mbox{ord}_p(|\mathfrak{C}^+_{ns}(p)|) \ge [\frac{p}{4}]-1 + \mbox{irr}(p) $ where $ \mbox{irr}(p) $ is the index of irregularity of $ p $. }

\section{Order of growth of Cuspidal Divisor Class Groups}
 
We study the order of growth of $|\mathfrak{C}^+_{ns}(p)|  $:  

\begin{thm}\label{Ordinedigrandezza}
If $ p\equiv 1 $ mod $ 4 $:
$$ |\mathfrak{C}^+_{ns}(p)| = \mathcal{O} \left( \left( \frac{p}{2\sqrt{6}} \right)^{p-4}  \right). $$
If $ p \equiv 3 $ mod $ 4 $:
$$ |\mathfrak{C}^+_{ns}(p)| = \mathcal{O} \left(\left( \frac{p}{2\sqrt[4]{90}} \right)^{p-4}\right). $$
Furthermore for every $ p $ we have: 
$$  |\mathfrak{C}^+_{ns}(p)| = \Omega \left( \left(\frac{p}{2\pi}\right)^{p-4} \right) $$  
so
$$ \ln |\mathfrak{C}^+_{ns}(p)| - p\ln p = \Theta(p).  $$
\end{thm}  
 
\begin{proof}
Let $ T: \mathbb{F}_{p^2} \rightarrow \mathbb{F}_p $ a surjective $ \mathbb{F}_p$-linear map. Let $ \chi $ be a multiplicative character on $ \mathbb{F}_{p^2}^*$. Following \cite[Paragraph 1.5]{KL} and \cite{lfunctions} we define for $ s \in \mathbb{C} $ the generalized $L $-series: 
$$ L(s,\chi,T)= p^{-s-1} \sum_{\alpha \in \mathbb{F}_{p^2}} \chi(\alpha) \zeta \left( s,  \left\langle \frac{T(\alpha)}{p} \right\rangle \right) $$
where $ \zeta(s,u) $ is the Hurwitz zeta function which is defined for a real number $ 0 < u \le 1 $ by:
$$ \zeta(s,u)= \sum_{n=0}^{\infty}\displaystyle\frac{1}{(n+u)^s}.  $$

\noindent 
By a classical result of Hurwitz \cite[Theorem 4.2]{Washington} we have: 
$$ \zeta(1-m,u) = -\displaystyle \frac{1}{m} B_m(u) $$
\noindent
so if $ \chi $ is an even character and $ T(\alpha):= \frac{1}{2}\mbox{Tr}(\alpha) \mbox{ mod }p $,  we have:
$$ B_{2,\chi} = -L(-1,\chi,T).  $$
Consider the following relation (cfr. \cite[Theorem 5.2, Chapter 1]{KL}): 
$$ L(s,\chi,T)= \frac{1}{2  \pi  p i} \left(\frac{2\pi}{p}\right)^s \Gamma(1-s)\tau(\chi,T)[e^{\frac{\pi i s}{2}}-\chi(-1)e^{-\frac{\pi i s}{2}}]L(1-s,\chi_{|_{\mathbb{F}_p}}) $$
where $ L(1-s,\chi_{|_{\mathbb{F}_p}}) $ is an ordinary $ L$-function and  
$$ \tau(\chi,T) = \sum_{\alpha \in \mathbb{F}_{p^2}} \chi(\alpha)e^\frac{2\pi i T(\alpha)}{p} $$ is a Gauss sum on $ \mathbb{F}_{p^2} $. 
From \cite[Proposition 11.5]{AnalyticIK} we have $|\tau(\chi,T)|=p $ so by virtue of Theorem \ref{robamia} we deduce:
$$    |\mathfrak{C}^+_{ns}(p)| = \Theta \left(  \left(\frac{p}{2\pi}\right)^{p-4} \prod_{j=1}^{\frac{p-3}{2}}|L(2,(\omega^{(2p+2)j})_{|_{\mathbb{F}_p}})| \right). $$   
If $ p\equiv 1 \mbox{ mod }4 $ let $ B $ the subgroup of squares of even characters mod $ p $.\\ 
In this case we have: 
$$\bigg{|} \displaystyle\prod_{j=1}^{\frac{p-3}{2}}L(2,(\omega^{(2p+2)j})_{|_{\mathbb{F}_p}})\bigg{|} = \bigg{|}\displaystyle\prod_{\chi\not=1 \mbox{  {\small even}}}  L(2,\chi^2)\bigg{|} = $$  $$ = \bigg{|}\zeta(2) \displaystyle\prod_{\chi\not=1 \mbox{, } \chi \in B }  L(2,\chi)^2\bigg{|} \le \zeta(2)^{\frac{p-3}{2}}, $$
$$ |\mathfrak{C}^+_{ns}(p)| = \mathcal{O} \left( \left( \frac{p}{2\sqrt{6}} \right)^{p-4} \right). $$
\noindent
If $ p\equiv 3 \mbox{ mod }4 $ we can obtain a more accurate estimation. In this case we have:  
$$ \displaystyle\prod_{j=1}^{\frac{p-3}{2}}L(2,(\omega^{(2p+2)j})_{|_{\mathbb{F}_p}}) = \displaystyle\prod_{\chi\not=1 \mbox{  {\small even}}}  L(2,\chi) $$ and by the arithmetic-geometric mean inequality:
$$ \displaystyle\bigg{|}\prod_{\chi\not=1 \mbox{  {\small even}}}  L(2,\chi)^2\bigg{|}^{\frac{2}{p-3}} \le \displaystyle\frac{2}{p-3} \sum_{\chi\not=1 \mbox{ {\small even}}}|L(2,\chi)|^2.  $$
For $ t\ge \frac{p+1}{2} $ let $ S(t,\chi)= \sum_{\frac{p+1}{2}\le n < t} \chi(n) $, then:
$$ L(2,\chi)= \sum_{n=1}^{\frac{p-1}{2}} \displaystyle\frac{\chi(n)}{n^2} + 2 \int_{\frac{p+1}{2}}^{\infty} \frac{ S(t,\chi)}{t^3} dt .$$
From the Polya-Vinogradov inequality \cite[Theorem 12.5]{AnalyticIK} for every $ \chi\not=1 $ we have $ |S(t,\chi)|\le 6\sqrt{p} \ln p $ and consequently: 
$$ |L(2,\chi)| \le \bigg{|}\sum_{n=1}^{\frac{p-1}{2}}  \displaystyle\frac{\chi(n)}{n^2} \bigg{|} + \frac{24 \sqrt{p} \ln p}{(p+1)^2} . $$  
By the triangle inequality we obtain:
$$  \displaystyle\bigg{(}\sum_{\chi\not=1 \mbox{ {\small even}} } |L(2,\chi)|^2\bigg{)}^{\frac{1}{2}} \le \bigg{(} \sum_{\chi \not=1 \mbox{ {\small even}} }  \bigg{|}\sum_{n=1}^{\frac{p-1}{2}}  \displaystyle\frac{\chi(n)}{n^2} \bigg{|}^2 \bigg{)}^{\frac{1}{2}} + \bigg{(} \sum_{\chi \not=1 \mbox{ {\small even}}} \left(  \frac{24 \sqrt{p} \ln p}{(p+1)^2} \right)^2  \bigg{)}^{\frac{1}{2}} 
 $$
$$ \le  \bigg{(} \sum_{\chi {\small\mbox{ even }}}  \bigg{|}\sum_{n=1}^{\frac{p-1}{2}}  \displaystyle\frac{\chi(n)}{n^2} \bigg{|}^2 \bigg{)}^{\frac{1}{2}} +  \frac{24 \sqrt{p} \ln p}{(p+1)^2} \sqrt{\frac{p-3}{2}} $$
$$ \le \pi^2 \sqrt{\frac{p-1}{180}} + \frac{24 \sqrt{p} \ln p}{(p+1)^2} \sqrt{\frac{p-3}{2}}  $$
because:
$$  \sum_{\chi {\small\mbox{ even }}}  \bigg{|}\sum_{n=1}^{\frac{p-1}{2}}  \displaystyle\frac{\chi(n)}{n^2} \bigg{|}^2 =  \sum_{\chi {\small\mbox{ even }}}\sum_{n=1}^{\frac{p-1}{2}}\sum_{m=1}^{\frac{p-1}{2}}\frac{\chi(n)\overbar{\chi}(m)}{n^2 m^2} = \frac{p-1}{2} \sum_{n=1}^{\frac{p-1}{2}} \frac{1}{n^4} \le \frac{p-1}{2} \zeta(4)  $$
and $  \sum_{\chi {\small\mbox{ even }}}\chi(n)\overbar{\chi}(m) =0 $ except when $ n\equiv \pm m \mbox{ mod } p $. 
Hence:
$$ \displaystyle\bigg{|}\prod_{\chi\not=1 \mbox{  {\small even}}}  L(2,\chi)^2\bigg{|}^{\frac{2}{p-3}} \le \frac{2}{p-3} \left( \pi^2 \sqrt{\frac{p-1}{180}} + \frac{24 \sqrt{p} \ln p}{(p+1)^2} \sqrt{\frac{p-3}{2}}\right)^2 , $$
$$ \frac{2}{p-3} \left( \pi^2 \sqrt{\frac{p-1}{180}} + \frac{24 \sqrt{p} \ln p}{(p+1)^2} \sqrt{\frac{p-3}{2}}\right)^2   = \frac{\pi^4}{90} + \mathcal{O}\left(\frac{1}{p} \right), $$ 
$$ \displaystyle\bigg{|}\prod_{\chi\not=1 \mbox{  {\small even}}}  L(2,\chi)\bigg{|} = \mathcal{O} \left( \left(\displaystyle\frac{\pi}{  \sqrt[4]{90}}\right)^{p} \right),  $$ 
$$ |\mathfrak{C}^+_{ns}(p)| = \mathcal{O} \left(\left( \frac{p}{2\sqrt[4]{90}} \right)^{p-4}\right). $$
   \noindent 
Let $ \Lambda $ be the von Mangoldt function:
$$\Lambda(n) = \begin{cases} \log p & \text{if }n=p^k \text{ for some prime } p \text{ and integer } k \ge 1, \\ 0 & \text{otherwise.} \end{cases}  $$
From the classical relation:
$$ \ln L(s,\chi) = \sum_{n=2}^{\infty}\frac{\Lambda(n)}{\ln n}\chi(n)n^{-s}$$ 
we have that if $ p \equiv 1 \mbox{ mod }4$:
$$ \prod_{\chi \in B}  L(2,\chi) = \mbox{exp}\left( \sum_{\chi \in B}\sum_{n=2}^{\infty}\frac{\Lambda(n)}{n^2\ln n}\chi(n) \right)= $$  $$ =\mbox{exp}\left( \frac{p-1}{4}\sum_{ \begin{scriptsize} \begin{array}{c} n=2 \\ n^4 \equiv 1 \mbox{ mod }p \end{array} \end{scriptsize} } ^{\infty}\frac{\Lambda(n)}{n^2\ln n} \right) \ge 1 $$
\noindent
and analogously if $ p \equiv 3 \mbox{ mod }4 $ we have: 
$ \prod_{\chi \mbox{  {\small even}}}  L(2,\chi) \ge 1 .$\\
Alternatively, we could notice that if $X$ is a group of Dirichlet characters and $K$ is the associated field with ring of integers $ \mathcal{O}_K $, from \cite[Theorem 4.3]{Washington} we have:
$$ \prod_{\chi \in X}L(2,\chi)= \zeta_K (2) = 1 + \sum_{I \subsetneq  \mathcal{O}_K} \frac{1}{[ \mathcal{O}_K: I]^2} > 1.$$
So we can easily deduce: 
$$  |\mathfrak{C}^+_{ns}(p)| = \Omega \left( \left(\frac{p}{2\pi}\right)^{p-4} \right) $$    
and
$$ \ln |\mathfrak{C}^+_{ns}(p)| - p\ln p = \Theta (p).  $$
  \end{proof}

\section{Eigencomponents at prime level}

Following \cite{pprimarypart}, in order to study the $ p-$primary part $ \mathcal{C}_p $ of $ \mathfrak{C}^+_{ns}(p) $ it is convenient to define:
$$ R_p := \mathbb{Z}_p[H] \mbox{ with } H=(\mathbb{Z}/p\mathbb{Z})^*/(\pm 1),$$
$$ R_{p,0}:= \{x \in R_p \mbox{ of degree } 0\}, $$
where the degree of $ x=\sum_{h \in H}x_h h $ is defined by $ \deg x=\sum_{h \in H}x_h $.
Of course we have $ \mathcal{C}_p \cong R_{p,0}/R_p \theta $ because when $ p \ge 5 $, the Stickelberger element $ \theta $ belongs to $ \frac{1}{12}\mathbb{Z}[H] $ and 12 is invertible in $\mathbb{Z}_p $. We have the following decomposition:
 $$ \mathbb{Z}_p \otimes \mathcal{C}_p = \mathop{\bigoplus_{\chi} } \mathcal{C}_p(\chi), $$
\noindent
where $ \chi$ ranges over the non trivial characters:
 $$ \chi: (\mathbb{Z}/p\mathbb{Z})^*/(\pm 1) \rightarrow \mathbb{Z}^*_p $$
\noindent
and $ a \in \mathcal{C}_p(\chi) $ if and only if $ a \cdot b = \chi(b) \cdot a $ for every $ b \in \mathbb{Z}_p \otimes \mathcal{C}_p $. \\
Let $ w $ be a generator of $ H=(\mathbb{Z}/p\mathbb{Z})^*/\{\pm 1\} $. Notice that:  
\noindent
$$ \chi(\theta) =\displaystyle\frac{p}{2} \sum_{i=0}^{\frac{p-3}{2}}\sum_{ { \begin{scriptsize} \begin{array}{c} x \in (\mathbb{F}^*_{p^2}/\{\pm 1\})  \\  \pm x^{p+1} =w^i \end{array} \end{scriptsize} }} B_2 \left( \left\langle \frac{\frac{1}{2}(\mbox{Tr}(x)) }{p} \right\rangle \right)\chi^{-1} (w^{i}) := S_{\chi^{-1}}, $$
so $ \theta $ operates on $ \mathcal{C}_p(\chi) $ as multiplication by $ S_{\chi^{-1}}  $ and consequently:
$$ \mathcal{C}_p(\chi)= \mathbb{Z}_p / S_{\chi^{-1}} \mathbb{Z}_p .$$
We define the Teichm\"uller character 
$$ \omega: {\mathbb{F}_p}^* \rightarrow  {\mathbb{Z}_p}^* $$
to be the character such that:
$$ \omega(a) = a \mbox{ mod } p. $$
Then we consider $ \phi = \omega^2 $ and view it as a charcater on $ H $. 
\begin{prop}\label{langata} Define:
$$ B'_{2,\phi^j}:= B_{2,\omega^{4j}} =  p \sum_{a=1}^{p-1}\phi^{2j}(a)B_2\left(\frac{a}{p}\right). $$
If $ p \equiv 1 \mbox{ mod } 4 $ and $ 1 \le j \le \frac{p-5}{4} $:
$ \mbox{ ord}_p S_{\phi^j}= 1+ \mbox{ ord}_p B'_{2,\phi^j} \ge 1$. \\
If $ p \equiv 1 \mbox{ mod } 4 $ and $ j = \frac{p-1}{4} $:
$ \mbox{ ord}_p S_{\phi^{\frac{p-1}{4}}}=0 $. \\
If $ p \equiv 1 \mbox{ mod } 4 $ and $ \frac{p+3}{4} \le j \le \frac{p-3}{2} $:
$ \mbox{ ord}_p S_{\phi^j}=  \mbox{ ord}_p B'_{2,\phi^j} \ge 0$.\\
\noindent If $ p \equiv 3 \mbox{ mod } 4 $ and $ 1 \le j \le \frac{p-7}{4} $:
$ \mbox{ ord}_p S_{\phi^j}= 1+ \mbox{ ord}_p B'_{2,\phi^j} \ge 1$. \\
If $ p \equiv 3 \mbox{ mod } 4 $ and $ j = \frac{p-3}{4} $:
$ \mbox{ ord}_p S_{\phi^{\frac{p-3}{4}}}= 1+\mbox{ ord}_p B'_{2,\phi^{\frac{p-3}{4}}} = 0 $. \\
If $ p \equiv 3 \mbox{ mod } 4 $ and $ \frac{p+1}{4} \le j \le \frac{p-3}{2} $:
$ \mbox{ ord}_p S_{\phi^j}=  \mbox{ ord}_p B'_{2,\phi^j} \ge 0$.
 \end{prop}
 \begin{proof}
Corollary of Theorems 4.2 and 4.3 of \cite[Chapter 5]{KL}.
\end{proof}

\noindent
It is immediate to deduce that $ \mbox{ord}_p(|\mathfrak{C}^+_{ns}(p)|)\ge [\frac{p}{4}]-1 $. From \cite[Theorem 5.16]{Washington} we recall that a prime is regular (i.e. does not divide $h^-_p  $, the relative class number of the cyclotomic fields $ \mathbb{Q}(\zeta_p) $), if and only if $ p $ does not divide the numerator of any of the Bernoulli numbers $ b_n $ for $ n =2,4,6,...,p-3 $. We propose an analogue for the modular case of Mazur-Wiles  \cite[Chapter 13]{Washington} and Herbrand-Ribet theorems \cite[Chapters 6 and 15]{Washington} for cyclotomic fields.  Usually we expect $ \mathcal{C}_p \cong (\mathbb{Z}/p\mathbb{Z})^{[\frac{p}{4}]-1} $ and exceptions occur only when $p$ is an irregular prime.

\begin{thm}\label{analogomodulare}  $ \mbox{ord}_p(|\mathfrak{C}^+_{ns}(p)|)= [\frac{p}{4}]-1 $ if and only if $ p $ is a regular prime or $ p \equiv 1 \mbox{ mod }4 $ is irregular and $ p $ does not divide the numerator of any Bernoulli number $ b_{4j+2} $ for $ j \le \frac{p-5}{4}$. \\
If  $ \mbox{ord}_p(|\mathfrak{C}^+_{ns}(p)|)= [\frac{p}{4}]-1 $ we have $ \mathcal{C}_p \cong (\mathbb{Z}/p\mathbb{Z})^{[\frac{p}{4}]-1} $.\\
If $ p \equiv 1 \mbox{ mod } 4 $ and  $ \mbox{ord}_p(|\mathfrak{C}^+_{ns}(p)|)> [\frac{p}{4}]-1 $ then $ \mathfrak{C}^+_{ns}(p) $ contains an element of order $ p^2 $ .\\
If $ p \equiv 3 \mbox{ mod } 4 $ and  $ \mbox{ord}_p(|\mathfrak{C}^+_{ns}(p)|)> [\frac{p}{4}]-1 $ then $ \mathfrak{C}^+_{ns}(p)$ contains  an element of order $ p^2 $ if and only if $ p $ divides $ b_{4j+2} $ for some $ j \le \frac{p-7}{4} $.\\  
If $ p \equiv 3 \mbox{ mod } 4 $ then $  \mbox{ord}_p(|\mathfrak{C}^+_{ns}(p)|) \ge [\frac{p}{4}]-1 + \mbox{irr}(p) $ where $ \mbox{irr}(p) $ is the index of irregularity of $ p $.

\end{thm}    
 
\begin{proof}

Let $ 1 \le a \le p-1 $. From Proposition \ref{langata}, we need to investigate when $ p $ divides $ B'_{2,\phi^j} $.
We will provide  an alternative proof of Von Staudt type congruences \cite[p.125]{KL}:
$$ \displaystyle\frac{1}{n}B_{n,\omega^{k-n}} = \displaystyle\frac{1}{k}B_k \mbox{ mod } p  $$
in the case $ k=2$. For the general case see \cite[Chapter 2, Theorem 2.5]{Langcf}. We have:
$$  B'_{2,\phi^j} = \frac{p}{6} \sum_{i=1}^{p-1}\omega^{4j}(a) - \sum_{i=1}^{p-1}a \omega^{4j}(a) + \frac{1}{p} \sum_{i=1}^{p-1}a^2 \omega^{4j}(a). $$ 
\noindent 
But $ \sum_{i=1}^{p-1}a \omega^{4j}(a) \equiv \sum_{i=1}^{p-1}a^{4j+1} \equiv 0 \mbox{ mod }p$, so $  pB'_{2,\phi^j} \equiv  \sum_{i=1}^{p-1}a^2 \omega^{4j}(a) \mbox{ mod }p^2 $.
If $ p \equiv 3 \mbox{ mod } 4 $ and $ j=\frac{p-3}{4} $ we have $ \sum_{i=1}^{p-1}a^2 \omega^{p-3}(a) \equiv -1 \mbox{ mod } p $ and $\mbox{ord}_p B'_{2,\phi^{\frac{p-3}{4}}} = -1 $. Apart from this exception we have that $p|\sum_{i=1}^{p-1}a^2 \omega^{4j}(a)$ and $\mbox{ord}_p B'_{2,\phi^j} \ge 0 $.
\\Let $ \omega_1(a)\in \mathbb{Z} $ with $ 1 \le \omega_1(a) \le p-1 $ chosen so that we have $ \omega(a) = a + \omega_1(a)p \mbox{ mod } p^2$.  Since $ \omega(a)^p=\omega(a) $ we deduce $ \omega_1(a) \equiv \frac{a^p-a}{p} \mbox{ mod }p$. Ergo:
$$  pB'_{2,\phi^j} \equiv \sum_{i=1}^{p-1} a^2(a + \omega_1(a)p)^{4j} \mbox{ mod } p^2 $$ 
$$ \equiv \sum_{i=1}^{p-1}a^2(a^{4j} + 4ja^{4j-1}p\omega_1(a)) \mbox{ mod }p^2 $$
$$ \equiv \sum_{i=1}^{p-1}a^2(a^{4j} + 4ja^{4j-1}(a^p-a)) \mbox{ mod }p^2  $$
$$ \equiv (1-4j)\sum_{i=1}^{p-1}a^{4j+2} + 4j\sum_{i=1}^{p-1}a^{4j+1+p} \mbox{ mod } p^2.$$
\noindent  
Let $ B_n(x) $ the $ n$-th Bernoulli polynomial. From Faulhaber's formula (cf. \cite[Theorem 1.5]{Bernoulli} and \cite[Proposition 9.2.12]{Cohen}) we have: 
$$ pB'_{2,\phi^j} \equiv (1-4j)\frac{B_{4j+3}(p)-B_{4j+3}(0)}{4j+3}+4j\frac{B_{4j+2+p}(p)-B_{4j+2+p}(0)}{4j+2+p} \mbox{ mod }p^2. $$
But $ B_n(x)= \sum_{h=0}^{n}\binom{n}{h}b_{n-h}x^h $ so we obtain:
$$  pB'_{2,\phi^j} \equiv p(1-4j)b_{4j+2}+ 4pj b_{4j+1+p} \mbox{ mod }p^2, $$ 
$$  B'_{2,\phi^j} \equiv (1-4j)b_{4j+2}+ 4j b_{4j+1+p} \mbox{ mod }p. $$
Notice that for $ 1 \le j \le \frac{p-3}{2} $, $ p-1 $ does divide neither $ 4j+2 $ nor $ 4j+1+p $ (we have already excluded the case $p \equiv 3 \mbox{ mod } 4$ and $ j=\frac{p-3}{4} $). We  
may apply Kummer's congruence \cite[Theorem 3.2]{Bernoulli}:  
$$ b_{4j+1+p} \equiv \frac{4j+1+p}{4j+2} b_{4j+2} \mbox{ mod }p.$$
So we have:  
 $$    B'_{2,\phi^j} \equiv b_{4j+2} \Big( 1-4j + 4j \frac{4j+1+p}{4j+2} \Big) \mbox{ mod }p.  $$ 
But $ \big(1-4j + 4j \frac{4j+1+p}{4j+2}\big) \equiv \frac{1}{2j+1} \mbox{ mod }p $, so $ p $ divides $ B'_{2,\phi^j} $ if and only if $ p$ divides $b_{4j+2}. $\\
If $p \equiv 1 \mbox{ mod }4 $ and $ \frac{p+3}{4} \le j \le \frac{p-3}{2}  $ we have that $  p $ divides $ b_{4j+2} $, if and only if $ p $ divides $ b_{4j+3-p}=b_{4 (j-\frac{p-1}{4})+2} $ so  $ \mbox{ord}_p(|\mathfrak{C}^+_{ns}(p)|)= [\frac{p}{4}]-1 $, if and only if $ p $ is regular or $ p $ is irregular, but $ p $ does not divide the numerator of any $b_{4j+2}$ for $ j \le \frac{p-5}{4} $.\\
If $ p \equiv 3 \mbox{ mod }4 $ and $ \frac{p+1}{4} \le j \le \frac{p-3}{2} $, we have that $ p $ divides $ b_{4j+2} $ if and only if $ p $ divides $ b_{4j+3-p} = b_{4(j-\frac{p-3}{4})} $. So in this case $ \mbox{ord}_p(|\mathfrak{C}^+_{ns}(p)|) > [\frac{p}{4}]-1 $, if and only if $ p $ is irregular. \\
The first claim is proved. The other assertions follow from Proposition \ref{langata}.
  
\end{proof}   

\noindent The table below is an excerpt of \cite[Table 8.1]{Io}:
\\
\begin{tabular}{rc}
\toprule
$ p $ &  $ |\mathfrak{C}^+_{ns}(p)|  $ \\
\midrule
23 & $ 23^4 \cdot 37181 $ \\
37 & $ 3^4 \cdot 7^2 \cdot 19^3 \cdot 37^8 \cdot 577^2 $ \\
43 & $ 2^2 \cdot 19 \cdot 29 \cdot 43^9 \cdot 463 \cdot 1051 \cdot 416532733 $ \\ 
59 & $ 59^{14} \cdot 9988553613691393812358794271  $ \\
67 & $  67^{16} \cdot 193 \cdot 661^2 \cdot 2861 \cdot 8009 \cdot 11287 \cdot 9383200455691459 $ \\
73 & $ 2^2 \cdot 3^4 \cdot 11^2 \cdot 37 \cdot 73^{17} \cdot 79^2 \cdot 241^2 \cdot 3341773^2 \cdot 11596933^2  $ \\
89 & $ 2^2 \cdot 3 \cdot 5 \cdot 11^2 \cdot 13^2 \cdot 89^{21} \cdot 4027^2 \cdot 262504573^2 \cdot 15354699728897^2  $\\
101 & $ 5^4 \cdot 17 \cdot 101^{24} \cdot 52951^2 \cdot 54371^2 \cdot 58884077243434864347851^2 $
\end{tabular} 
\\ 

\noindent The first four irregular primes are $ 37, 59, 67 $ and $ 101 $. Since $ 37|b_{32} $, $ 59|b_{44}, 67|b_{58} $ and $ 101| b_{68} $ according to Theorem \ref{analogomodulare} we immediately deduce that:  $$ \mathcal{C}_{37} \cong (\mathbb{Z}/37\mathbb{Z})^8 \mbox{,  } \mathcal{C}_{101} \cong (\mathbb{Z}/101\mathbb{Z})^{24} $$   and $ \mbox{ord}_{p}(|\mathfrak{C}^+_{ns}(p)|) >  [\frac{p}{4}]-1 $ for $ p=59,67$. Moreover, knowing from explicit calculation that $ \mbox{ord}_{59}(|\mathfrak{C}^+_{ns}(59)|)= 14 $, $ \mbox{ord}_{67}(|\mathfrak{C}^+_{ns}(67)|)= 16 $, we may conclude: $$ \mathcal{C}_{59} \cong (\mathbb{Z}/59\mathbb{Z})^{14} \mbox{ and } \mathcal{C}_{67} \cong (\mathbb{Z}/67\mathbb{Z})^{14} \times (\mathbb{Z}/67^2\mathbb{Z}).  $$  
 
Let $ q $ a prime that does not divide $ p(p^2-1) $. Let $ n >0 $ be the order of $ q \mbox{ mod }\frac{p-1}{2} $ and let $ \mathfrak{o}_{q,n} $ be the ring of integers in the unramified extension of the $ q-$adic field $ \mathbb{Q}_q$ of degree $ n $. We have an analogous decomposition for the $q-$primary part $ \mathcal{C}_{p,q} $ of $ \mathfrak{C}^+_{ns}(p) $: 
$$ R_{p,q} := \mathbb{Z}_q[H] \mbox{ with } H=(\mathbb{Z}/p\mathbb{Z})^*/(\pm 1), $$  $$ R_{p,q,0}:= \{x \in R_{p,q} \mbox{ of degree } 0\} \mbox{,}$$
$$ \mathcal{C}_{p,q} \cong R_{p,q,0}/ R_{p,q} \theta \mbox{ and } \mathfrak{o}_{q,n} \otimes_{\mathbb{Z}_q} \mathcal{C}_{p,q} = \mathop{\bigoplus_{\chi} } \mathcal{C}_{p.q}(\chi) $$
\noindent
where $ \chi$ ranges over the non trivial characters:
 $$ \chi: (\mathbb{Z}/p\mathbb{Z})^*/(\pm 1) \rightarrow \mathfrak{o}^*_{q,n}. $$ 
 
\begin{prop} Let $ p \equiv 1 \mbox{ mod }4 $ and $ q \ge 7 $ a prime different from p that does not divide $ p^2-1$. If $ q^m $ is the maximal $q$-power dividing  $ |\mathfrak{C}^+_{ns}(p)| $ then $ m $ is even. 
\end{prop}  
 
\begin{proof} We have $ \mathcal{C}_{p,q}(\chi) = \mathfrak{o}_{q,n} / \mathfrak{o}_{q,n} \chi(\theta) $ and from Theorem 4.5 of \cite[Chapter 5]{KL} if $ p \equiv 1 \mbox{ mod }4 $ and $ \chi^2_1 = \chi_2^2 $ we have $ \mbox{ord} _q \chi_1(\theta) =\mbox{ord} _q \chi_2(\theta)  $. If $ \chi^2=1 $ we have $ \mathcal{C}_{p,q}(\chi)=0 $.
\end{proof}

\end{document}